\documentclass[12pt,notitlepage]{amsart}
\usepackage{latexsym,amsfonts,amssymb,amsmath,amsthm} 
\newcommand{\N}{\ensuremath{\mathbb N}}

\newcommand{\Z}{\ensuremath{\mathbb Z}}
\newcommand{\R}{\ensuremath{\mathbb R}}

\newcommand{\Q}{\ensuremath{\mathbb Q}}
\newcommand{\C}{\ensuremath{\mathbb C}}

\theoremstyle{plain}		
	\newtheorem{theorem}{Theorem}[section]

	\newtheorem{cor}[theorem]{Corollary}
    
     \newtheorem{lemma}[theorem]{Lemma}

\theoremstyle{remark}		
	
	\newtheorem*{remark}{Remark}

\numberwithin{equation}{section}
\begin{document}
\title{The Functional Equation and Beyond Endoscopy}
\author{P. Edward Herman} 
\address{University of Chicago}
\email{peherman@math.uchicago.edu}

\begin{abstract}
  In his paper ``{\it Beyond  Endoscopy,}" Langlands tries to understand functoriality via poles of L-functions.  The following paper further investigates the analytic continuation of a L-function associated to a $GL_2$ automorphic form through the trace formula. Though the usual way to obtain the analytic continuation of an L-function is through its functional equation, this paper shows that by simply assuming the trace formula, the functional equation of the L-function may be recovered. This paper is a step towards understanding the analytic continuation of the L-function at the same time as capturing information about functoriality. 
 
 From an analytic number theory perspective, obtaining the functional equation from the trace formula implies that Voronoi summation should in general be also a consequence of the trace formula.
 
   \end{abstract}

\dedicatory{Dedicated to the memory of  Jonathan Rogawski}

\maketitle

\section{ Beyond Endoscopy}

        Let $\mathbb{A_Q}$ be the ring of adeles of $\Q,$ and $\pi$ be
    an automorphic cuspidal representation of
    $GL_2(\mathbb{A_Q}).$ We define $m(\pi,\rho)$ to be the order of the
    pole at $s=1$ of $L(s,\pi,\rho),$ where $\rho$ is a
    representation of the dual group $GL_2(\C).$

        Langlands proposes the study of \begin{equation} \label{p_sum} \lim_{X \to \infty} \sum_\pi \frac{1}{X}
        tr(\pi)(f)\sum_{p\leq X} \log (p) a(p,\pi,\rho).\end{equation} Here $f$
        is a nice test function on $GL_2(\mathbb{A_Q}),$ and
        $tr(\pi)(f)$ is the trace of the operator defined by $f$
        on $\pi.$ $a(p,\pi,\rho)$ is the $p-$th Dirichlet
        coefficient of $L(s,\pi,\rho).$ The quantity $$\lim_{X \rightarrow
        \infty}\frac{1}{X}\sum_{p\leq X} \log (p) a(p,\pi,\rho),$$
        is equal to $m(\pi,\rho).$

            Therefore, summing over the range of representations
            $\pi$  will project only on to the ones which have nontrivial multiplicity. The tool used to study this sum
            over the spectrum of forms $\pi$ is the trace formula.
            Ultimately, one gets from use of the trace formula a sum over
            primes and conjugacy classes, and hopes by analytic
            number theory techniques to take the limit. One hopes
            that after getting the limit, one can decipher and
            construct the L-functions having non-trivial
            multiplicity of the pole at $s=1.$ 
            Sarnak addresses (\ref{p_sum}) in \cite{S} for $\rho=std$
            the standard representation. He points out that such a
            computation can be done, but the tools used for the
            study of sums of primes is limited, and this problem
            is perhaps more tractable if rather studied over the
            sum of integers.

 Sarnak's idea then is to evaluate \begin{equation} \label{nsum} \lim_{X
\to \infty} \sum_\pi \frac{1}{X}
        tr(\pi)(f)\sum_{n\leq X}  a(n,\pi,\rho).\end{equation}
        This should ``detect," rather than the multiplicities of the
        poles, the residue of the poles of the associated
        L-functions. As well, instead of using the Arthur Selberg trace formula, he uses the Petersson-Kuznetsov trace formula, which is a special case of the relative trace formula (\cite{KL}).  One advantage to this trace formula is that the spectrum contains only generic representations, so we avoid the task of excising the trivial representation as in \cite{FLN}. As well, the geometric side of the relative trace formula has a nice ``streamlined" appearance as a sum of Kloosterman sums. This is in comparison to the Arthur-Selberg trace formula which has orbital integrals associated to different conjugacy classes for which analysis of each class could be different.
        
         The disadvantage to the relative trace formula is that each automorphic representation $\pi$ on the spectral side of the trace formula is weighted by a factor $L(1,\pi,ad)^{-1},$ which is the adjoint representation of $\pi$ evaluated at $s=1.$ This can perhaps make matching two different trace formulas more difficult. Another disadvantage of using the relative trace formula is that the Arthur-Selberg trace formula is in much better shape to generalize to other groups. Namely, one now has full use of the stable trace formula due to the proof of the Fundamental Lemma by Ngo, \cite{N}. With the stable trace formula, one can compare stable conjugacy classes for different groups (specifially endoscopic groups), from which one can then compare automorphic representations for the respective groups. 
        
        However in our case of studying $GL(2),$ the disadvantages seem minimal, and in fact the crucial exponential sums one encounters in either trace formula are the same. Sarnak in \cite{S} made some points on the essential differences of the geometric sides of the two trace formulas. Also, in the case of $GL(2),$ the stable trace formula is the same as the Arthur-Selberg trace formula, so one should not expect an advantage of one trace formula over another. 
        
        \subsection{Sarnak's analysis for $\rho=std$}
        
             The obvious first example to test Langlands's beyond endoscopy idea is for $\rho$ the standard representation. In this case we do not expect the L-functions to have any poles except for the continuous spectrum, but in this case there are not any poles as the spectrum is not spectrally isolated. So we expect in the case of $\rho=std$ that \begin{equation} \label{eq:nsum00} \lim_{X
\to \infty} \sum_\pi \frac{1}{X}
        tr(\pi)(f)\sum_{n\leq X}  a(n,\pi,std)=0.\end{equation} 
        
        Rather than use the adelic language, Sarnak uses the classic
        Petersson-Kuznetsov trace formula. To go from \eqref{eq:nsum00} to a classic approach, one can follow the great expository of Rogawski \cite{Ro} or the book of Knightly and Li, \cite{KL1}. Then for an automorphic form $f$ with normalized Fourier coefficients $a_n(f)$ associated to a representation $\pi,$ Sarnak \cite{S} showed, up to some weight factors needed in the trace formula,  \begin{equation} \label{eq:sarrk} \sum_{n \leq X} \sum_f a_n(f)g(n/X)=O(X^{-A})\end{equation} for any $A>0.$ Here $X$ is a large fixed parameter and $g\in C_0^{\infty}(\R^{+})$ is used for ``smoothing" the $n$-sum. Why is this smoothing needed? It is certainly not essential, but when one goes to the geometric side of the trace formula to get the bound \eqref{eq:sarrk}, one requires freedom to apply analytic manipulations (interchanging sums, Fourier transforms, etc..). With the smoothing function $g,$ these problems are removed and one can focus on the central issue of the arithmetic, which is the true difficulty in these problems.  One can recover the left hand side of \eqref{eq:nsum00} by applying techniques in \cite{I}. For completeness, we will reproduce Sarnak's argument in the appendix.

        \subsection{Results of the paper}
        
       Clearly \eqref{eq:sarrk} is a stronger result than \eqref{eq:nsum00}, and up to using Hecke operators, is equivalent to $L(s,f)=\sum_{n=1}^{\infty}\frac{a_n(f)}{n^s}$ having analytic continuation to the complex plane. We see the analytic continuation of the left hand side of \eqref{eq:sarrk} by Mellin inversion. By applying Mellin inversion to \eqref{eq:sarrk} we get \begin{equation}\label{eq:ini}
\frac{1}{X}\sum_f \sum_{n \leq X}g(n/X)a_n(f)=
\frac{1}{2\pi i} \int_{\sigma -i\infty}^{\sigma+i\infty} G(s)[\sum_f L(s,f)]X^s ds, 
\end{equation}
where $G(s)=\int_0^\infty g(x)x^{s-1}dx$ is the Mellin transform with $\sigma >2$ to ensure the convergence of the integral. Now using the right hand side of \eqref{eq:sarrk} we know that the contour in \eqref{eq:ini} can be shifted (using decay properties of $G(s)$) to $\sigma=-A, A >0.$ 
So in Sarnak's application of the trace formula to get \eqref{eq:sarrk} we indirectly applied a functional equation of the L-function for each automorphic form $f$ in our spectral sum. Can we actually see directly the functional equation via manipulations on the geometric side of the trace formula?  In other words, can we show directly via the trace formula that $$ \frac{1}{2\pi i} \int_{\sigma -i\infty}^{\sigma+i\infty} G(s)[\sum_f L(s,f)]X^s ds=  \frac{1}{2\pi i} \int_{\sigma -i\infty}^{\sigma+i\infty} G(s)[\sum_f\frac{i^k\gamma(f,1-s)L(f,1-s)}{\gamma(f,s)}]X^sds?$$ We will prove this equality and get the functional equation for a fixed automorphic form $f$ in this note. Of course there are much more easy ways to get the functional equation for a $GL(2)$ automorphic form, but in consideration of Langlands's beyond endoscopy idea, a trace formula approach seems the most systematic way to get analytic continuation for all L-functions $L(s,\pi,\rho)$ associated to a dual group representation $\rho$ of an automorphic representation $\pi$ of a group $G.$ This is certainly a more difficult question then investigating whether the L-function has a pole at $s=1$ or not.
\subsubsection{Voronoi Summation}
If one can always recover the functional equation from the trace formula, then from the analytic number theory perspective, the Voronoi summation should be implied also from the trace formula. 
 For example in the papers \cite{KMV},\cite{KMV1}, an application of a trace formula and a Voronoi summation are used to get results on subconvexity.  Could one avoid Voronoi summation and just  apply the trace formula? In the paper \cite{H3}, we do just that to get subconvexity for the Rankin-Selberg L-function in both levels by applying a double trace formula instead of a Voronoi summation and a single trace formula.

\subsection{Key steps in proof}

As for the proof of the main theorem, one sees the role of the sum over the Kloosterman sums on the geometric side of the trace formula interacts with the averaging coming from the Dirichlet series for the standard L-function.  

To see the functional equation of a $GL_2$ L-function, the Dirichlet series sum exchanges roles with the sum of Kloosterman sums. There are two important steps in this switching of roles of parametrization. One is elementary reciprocity, $$\frac{\overline{A}}{B}+ \frac{\overline{B}}{A}\equiv \frac{1}{AB}(1)$$ which allows one to invert the modulus of exponential sums. This simple reciprocity seems to come up several times in these beyond endoscopy calculations (see e.g. \cite{H1}, \cite{H2}). The second important tool is the integral representation $$\int_0^\infty \exp(-\alpha x) J_{\nu}(2\beta\sqrt{x})J_{\nu}(2\gamma\sqrt{x})dx=\frac{1}{\alpha}I_{\nu}(\frac{2\beta \gamma}{\alpha})\exp(\frac{-(\beta^2 + \gamma^2)}{\alpha})dx.$$ Reminded that Bessel functions are the archimedean version of Kloosterman sums, this representation implies that a Fourier transform of a product of Kloosterman sums is another Kloosterman sum times an exponential sum. It would be nice to see how these two steps are generalized for higher rank or for a relative trace formula for other groups.


{\bf Acknowledgements.}  
I want to dedicate this paper to my advisor, Jon Rogawski. I could not have asked for a better advisor than Jon as a graduate student at UCLA. His unfaltering patience and calm resolve balanced my personality, which was the opposite of patient in those days. When I would fail to understand an aspect of automorphic forms or the trace formula, Jon would dismiss my frustration and clarify the misunderstanding in a way that only Jon couldÑ with sympathetic composure and the knowledge of a veteran in the field.  While he was always collected when he addressed the challenges that I faced,  he was a passionate person and was ardent when he spoke about math. During times of stagnation, I would go to Jon to reinvigorate me. After talking to Jon, I always felt more inspired and confident.

From winter to the early part of summer of 2011, I was at the American Institute of Mathematics in Palo Alto and would visit Jon in Los Angeles every few months. I remember fondly going to coffee shops or to his home to discuss my new ideas as he shared his own. It was in one of these gatherings that he suggested how to isolate a single Hecke eigenform (found in Section \ref{rogawski}). I want to point this out because even today, Jon continues to inspire me. I am honored to have been one of his students and am also saddened that I was his last.

\section{Preliminaries}

We recall the functional equation for a cusp form. Let $D$ be a squarefree integer, $\chi$ be a primitive Dirichlet character modulo $D,$ and $k\geq 2, k\in 2\Z$ . Let $f \in S_k(D,\chi),$ where $S_k(D,\chi)$ is the space of holomorphic modular forms of weight $k$ and level $D$ with nebentypus $\chi,$ see \cite{IK}. In this case the space $S_k(D,\chi)$ can be spanned by an orthonormal basis of primitive newforms which we label $B_k(D,\chi).$ We note the Fourier coefficients $c_n(f)n^{\frac{k-1}{2}}$ of a form $f \in B_k(D,\chi)$ satisfy $$c_{n}
(f)c_{l}(f)=\sum_{r|(n,l)}\chi(r)c_{\frac{nl}{r^2}}(f)$$ for $(nl,D)=1,$ and also $|c_D(f)|=1.$ 

Let $L(f,s)=\sum_{n=1}^\infty \frac{c_n(f)}{n^s},$ and define  $\Lambda(f,s)=\gamma(f,s)L(f,s),$ where $$\gamma(f,s)=(\frac{\sqrt{D}}{2\pi})^{s} \Gamma(\frac{s+\frac{k-1}{2}}{2}) \Gamma(\frac{s+\frac{k+1}{2}}{2}).$$ The functional equation then says $\Lambda(f,s)=i^k\Lambda(f,1-s).$

The trace formula we use is Petersson's formula which is a variant of the relative trace formula \cite{KL}. 
This formula requires a normalization of the Fourier coefficients.  For $c_n(f)$ above, define $$a_n(f):=
   \frac{ \sqrt{\pi^{-k} \Gamma(k)}} {2^{k-1}}c_n(f).$$ Petersson's formula states 
    \begin{equation}\label{eq:pet} \sum_{f \in B_k(D,\chi)} a_n(f)\overline{a_l(f)}=\delta_{n,l}+ 2\pi i^{-k} \sum_{c\equiv 0(D)}^\infty \frac{S_{\chi}(n,l,c)}{c} J_{k-1}(\frac{4\pi \sqrt{nl}}{c}).\end{equation}
Here $$S_{\chi}(a,b,c)=\sum_{x(c)^{*}}\overline{\chi(x)}(\frac{ax +b\overline{x}}{c}),$$ where $x \overline{x} \equiv 1(c)$ and $e(x):=\exp(2\pi i x).$  $J_t(x)$ is the $J$-Bessel function with index $t$. 
 
 To relate the functional equation to the geometric side of the trace formula, we need an equivalent version of the functional equation for a form $f \in B_k(D,\chi)$ which is called Voronoi summation. The Voronoi summation needed is proved in the appendix of \cite{KMV}, and states
 \begin{theorem}\label{mic}
 Let $g \in C_0^\infty(\R^{+})$ and $f \in B_k(D,\chi),$ then for integers $a,c$ such that $(aD,c)=1,$ 
 \begin{equation}\label{eq:vor}
\sum_{n\geq 1} a_n(f)e(\frac{an}{c})g(n)= \frac{2\pi i^k \eta(f)\chi(-c)}{c\sqrt{D}} \sum_{n\geq 1} a_n(f_D)e(\frac{-n\overline{aD}}{c}) \int_0^\infty g(x)J_{k-1}(\frac{4\pi\sqrt{nx}}{\sqrt{D}c})dx,
\end{equation}
where $a\overline{a}\equiv 1 (c).$ Here $\eta(f)=\frac{\tau(\chi)}{a_D(f)\sqrt{D}},$ with $\tau(\chi)$ the Gauss sum associated to $\chi,$ and $$a_n(f_D)=\left\{ \begin{array}{ll}
        \overline{\chi(n)}a_n(f) & \text{if } (n,D)=1; \medskip \\
      \overline{a_n(f)} & \text{if } n|D^{\infty}.\end{array} \right.$$ \end{theorem}
 
 In our case, we only take $a=c=1.$ If so, the functional equation of the L-function $L(f,s)$ is equivalent to the Voronoi summation by using Mellin inversion on the left hand side of \eqref{eq:vor}, then applying the functional equation to $L(f,s)$ and using the integral representation
 $$J_{k-1}(x)=\frac{1}{4\pi i} \int_{(\sigma)}(\frac{x}{2})^{-s} \frac{\Gamma(\frac{s+\frac{k-1}{2}}{2})}{\Gamma(\frac{1-\frac{s}{2}+\frac{k-1}{2}}{2})}ds$$ for $0 < \sigma < 1,$ along with the duplication formula for the Gamma function.
 
\section{Main theorem}
 The main theorem of the paper is,
\begin{theorem}\label{mtheo}
Let $g \in C_0^\infty(\R^{+})$ with  $|x^j g^{(j)}(x)| \ll (1+|\log x|)$ and $(l,D)=1$ any positive integer, then only assuming Petersson's formula above, one gets
\begin{equation}\label{eq:mmm}\sum_{f \in B_k(D,\chi)}\overline{a_l(f)} \sum_{n\geq 1} a_n(f)g(n)= \sum_{f \in B_k(D,\chi)}\overline{a_{l}(f)}\left[ \frac{2\pi i^k\eta(f)}{\sqrt{D} }\sum_{n} a_n(f_D) \int_0^\infty g(x)J_{k-1}(\frac{4\pi\sqrt{nx}}{\sqrt{D}})dx\right].\end{equation}
\end{theorem}

 Using Hecke theory one gets, \begin{cor}\label{corr}
 For a modular form $f\in B_k(D,\chi),$ 
$$ L(f,s)=\frac{i^k\gamma(f,1-s)L(f,1-s)}{\gamma(f,s)},$$ or $$\Lambda(f,s)=\Lambda(f,1-s).$$
 \end{cor}
 
 \begin{proof}\{Theorem \ref{mtheo}\}
 
 Using Petersson's trace formula on the left hand side of \eqref{eq:mmm} one gets \begin{equation}
\label{eq:aft} \sum_{n}g(n)\left[\delta_{n,l}+ 2\pi i^{-k}\sum_{c=1}^\infty  \frac{S_{\chi}(n,l,Dc)}{Dc}J_{k-1}(\frac{4\pi\sqrt{nl}}{Dc})\right]=\end{equation} $$g(l)+2\pi i^{-k}\sum_{c=1}^\infty \sum_n g(n) \frac{S_{\chi}(n,l,Dc)}{Dc}J_{k-1}(\frac{4\pi\sqrt{nl}}{Dc}).$$

We can interchange the $c$-sum and $n$-sum as the latter is compactly supported. 

For now on we will ignore the term $g(l),$ and come back to it later. Opening up the Kloosterman sum and gathering the $n$-sum together, we apply Poisson summation on it in arithmetic progressions modulo $c$ getting  $$2\pi i^{-k}\sum_{c=1}^\infty\frac{1}{(Dc)^2} \sum_{x(Dc)^{*}}\chi(x)e(\frac{\overline{x}l}{Dc}) \sum_{m \in \Z} \sum_{k(Dc)}e(\frac{xk+mk}{Dc}) \int_{-\infty}^{\infty} g(t) J_{k-1}(\frac{4\pi \sqrt{t l}}{Dc}) e(\frac{-mt}{Dc})dt.$$ 
 
 Using \begin{equation}\label{eq:expp} \sum_{a(c)} e(\frac{ax}{c})= \left\{ \begin{array}{ll}
       c & \text{if } x\equiv 0(c) \\
        0 & \text{ else} \end{array} \right.,
 \end{equation}
 one gets
 \begin{equation}\label{eq:forget} 2\pi i^{-k}\sum_{c=1}^\infty\frac{1}{Dc} \sum_{\substack{m\neq 0 \in \Z \\ (m,Dc)=1}} \overline{\chi(m)}e(\frac{-l\overline{m}}{Dc}) \int_{-\infty}^{\infty} g(t) J_{k-1}(\frac{4\pi \sqrt{t l}}{Dc}) e(\frac{-mt}{Dc})dt.\end{equation} Note the $m=0$ disappears. 

 Now the interesting part of the argument is that the $c$-sum and $n$-sum swap roles, in that the $c$-sum will become part of the averaging coming from the L-function. 
 
We use the elementary reciprocity $$\frac{\overline{A}}{B}+ \frac{\overline{B}}{A}\equiv \frac{1}{AB}(1),$$ to get \begin{equation}\label{eq:afel}
2\pi i^{-k} \sum_{c=1}^\infty\frac{1}{Dc} \sum_{\substack{m\neq 0 \in \Z \\ (m,Dc)=1}} \overline{\chi(m)}e(\frac{l\overline{c}}{Dm})e(\frac{-l}{mDc}) \int_{-\infty}^{\infty} g(t) J_{k-1}(\frac{4\pi \sqrt{t l}}{Dc}) e(\frac{-mt}{Dc})dt.
\end{equation}

Moreso, the terms $m<0$ we write as $-m, m \in \N,$ and exchange sign to the $c$-sum. This can be clearly done everywhere except for the J-Bessel function and $\frac{1}{c}$ term. Using the fact that $J_{k-1}(-x)=-J_{k-1}(x),$  we can rewrite \eqref{eq:afel} as \begin{equation}\label{eq:negm}
2\pi i^{-k} \sum_{c \neq 0, c \in \Z }\frac{1}{Dc} \sum_{\substack{m=1 \\ (m,c)=1}} \overline{\chi(m)}e(\frac{l\overline{c}}{Dm})e(\frac{-l}{mDc}) \int_{-\infty}^{\infty} g(t) J_{k-1}(\frac{4\pi \sqrt{t l}}{Dc}) e(\frac{-mt}{Dc})dt.
\end{equation}
The rearrangement of the $m$-sum is accomplished by using a standard integration by parts argument in the $t$-integral and the estimate in the appendix of \cite{KMV}, $$ |z^k J_{\nu}(z)| \ll_{k,\nu} \frac{1}{(1+z)^{1/2}},$$ for $\Re(\nu) \geq 0.$ 

We also interchange the $c$-sum and $m$-sum. To justify the rearrangement, note for $c$ large, and using the power series expansion, we have the estimate $J_{k-1}(\frac{4\pi \sqrt{tl}}{Dc}) \ll \frac{1}{c^{k-1}}.$ Therefore for $N$ sufficiently large, estimating the exponentials and integral trivially  and noting $k \geq 2,$ we get
\begin{multline}
\sum_{c >N}\frac{1}{Dc} \sum_{\substack{m=1 \\ (m,c)=1}} \overline{\chi(m)}e(\frac{l\overline{c}}{Dm})e(\frac{-l}{mDc}) \int_{-\infty}^{\infty} g(t) J_{k-1}(\frac{4\pi \sqrt{t l}}{Dc}) e(\frac{-mt}{Dc})dt\ll    L(0,\overline{\chi}) \sum_{c >N}\frac{1}{c^k}  < \infty,
\end{multline}  $k \geq 2.$ Clearly, the $c$-sum up to $N$ is finite and is not a problem, and the sums can be interchanged.

Now we need a integral representation from \cite{GR}(6.615)

\begin{equation}\label{eq:gri}
\int_0^\infty \exp(-\alpha x) J_{\nu}(2\beta\sqrt{x})J_{\nu}(2\gamma\sqrt{x})dx=\frac{1}{\alpha}I_{\nu}(\frac{2\beta \gamma}{\alpha})\exp(\frac{-(\beta^2 + \gamma^2)}{\alpha})dx,
\end{equation}

 for $\Re(\nu) > -1.$
 
 We rewrite \eqref{eq:negm} as $$(2\pi i) (2\pi i^{-k}) \sum_{m} \frac{1}{m} \sum_{\substack{c\neq 0 , c \in Z \\ (c,m)=1}}  \overline{\chi(m)}e(\frac{l\overline{c}}{Dm}) \int_{-\infty}^{\infty} g(t) \left[\frac{m}{2\pi i Dc} J_{k-1}(\frac{4\pi \sqrt{t l}}{Dc}) e(\frac{-l}{mDc}) e(\frac{-mt}{Dc})\right]dt.$$
 
 Note the term in brackets is equal to the right hand side of \eqref{eq:gri} times $i^{k-1}$  for $\alpha=\frac{2\pi i Dc}{m}, \beta=\frac{2\pi \sqrt{l}}{m},$ and $\gamma=2\pi \sqrt{t}$ using the fact that for $k-1$ odd, $J_{k-1}(z)=i^{k-1} I_{k-1}(-iy).$

Using this integral representation, one has \begin{equation}\label{eq:afint} 4\pi^2  \sum_{m} \frac{1}{m} \sum_{\substack{c\neq 0 , c \in Z \\ (c,m)=1}}    \overline{\chi(m)}e(\frac{l\overline{c}}{Dm}) \int_{-\infty}^{\infty} g(t) \left[\int_0^\infty J_{k-1}(\frac{4\pi\sqrt{ly}}{m})J_{k-1}(4\pi\sqrt{ty})e(\frac{-Dcy}{m})dy\right]dt.
\end{equation}

We make a change of variables $y \to \frac{y}{D}$ to get \begin{equation}\label{eq:afintc} 4\pi^2  \sum_{m} \frac{1}{Dm} \sum_{\substack{c\neq 0 , c \in Z \\ (c,m)=1}}    \overline{\chi(m)}e(\frac{l\overline{c}}{Dm}) \int_{-\infty}^{\infty} g(t) \left[\int_0^\infty J_{k-1}(\frac{4\pi\sqrt{lDy}}{Dm})J_{k-1}(\frac{4\pi\sqrt{ty}}{\sqrt{D}})e(\frac{-cy}{m})dy\right]dt.
\end{equation}

Using $\frac{\tau(\chi)\tau(\overline{\chi})}{D}=1,$ we get \begin{equation}\label{eq:b4c} \frac{4\pi^2\tau(\chi)}{D} \sum_{m=1} \frac{1}{Dm} \sum_{c \in \Z }   \overline{\chi(m)}\tau(\overline{\chi})e(\frac{l\overline{c}}{Dm}) \int_{-\infty}^{\infty}J_{k-1}(\frac{4\pi\sqrt{lDy}}{Dm})\times \end{equation} $$ \left[\int_0^\infty g(t) J_{k-1}(\frac{4\pi\sqrt{ty}}{\sqrt{D}})dt\right] e(\frac{-cy}{m})dy.$$

 Anticipating using the Chinese remainer theorem we let $c'=Dc.$ So \eqref{eq:b4c} equals \begin{equation}\eqref{eq:b4c} \frac{4\pi^2\tau(\chi)}{D} \sum_{m=1} \frac{1}{Dm} \sum_{\substack{c'\in \Z, c'\equiv 0(D)\\ c'\neq 0, (c',m)=1} }   \overline{\chi(m)}\tau(\overline{\chi})e(\frac{l\overline{\frac{c'}{D}}}{Dm}) \int_{-\infty}^{\infty}J_{k-1}(\frac{4\pi\sqrt{lDy}}{Dm})\times \end{equation} $$ \left[\int_0^\infty g(t) J_{k-1}(\frac{4\pi\sqrt{ty}}{\sqrt{D}})dt\right] e(\frac{-c'y}{Dm})dy.$$
 
 We focus on the arithmetic inside the $c'$-sum.  We note using $(m,D)=1$ and the Chinese remainder theorem that \begin{multline}
 \overline{\chi(m)}\tau(\overline{\chi})e(\frac{l\overline{\frac{c'}{D}}}{Dm})= [\sum_{a(D)}\overline{\chi}(a)e(\frac{\overline
 {m}a}{D})][\sum_{\substack{b(m)^{*}\\\overline{b}l\equiv \frac{c'}{D}(m)}} e(\frac{b}{m})]= \sum_
 {\substack{x(Dm)\\ D\overline{x}l\equiv c'(Dm)}}\overline{\chi(x)}e(\frac{x}{Dm}).\end{multline}
Using \eqref{eq:expp} again the last line equals $$ \sum_
 {\substack{x(Dm)\\ D\overline{x}l\equiv c'(Dm)}}\overline{\chi(x)}e(\frac{x}{Dm})=\frac{1}{Dm} \sum_{x(Dm)}\overline{\chi(x)} e(\frac{x}{Dm}) \sum_{k(Dm)} e(\frac{k(Dl\overline{x}-c')}{Dm}).$$

Incorporating the above line and a rearrangement of the exponential sums, we have
\begin{equation}\label{eq:reg}
\frac{4\pi^2\tau(\chi)}{D} \sum_{m=1} \frac{1}{(Dm)^2}\sum_{c' \in \Z} \sum_{x(Dm)}\overline{\chi(x)} e(\frac{x}{Dm}) \sum_{k(Dm)} e(\frac{kDl\overline{x}}{Dm})  \int_{0}^{\infty}J_{k-1}(\frac{4\pi\sqrt{lDy}}{Dm})\times \end{equation} $$ \left[\int_0^\infty g(t) J_{k-1}(\frac{4\pi\sqrt{ty}}{\sqrt{D}})dt\right] e(\frac{-c'(y+k)}{Dm})dy. $$ We note the $c'$-sum has the restriction $c'\equiv 0(D)$ removed by the $k$-sum.
With a change of variables $y \to y - k,$ followed by $y \to Dmy,$ we get  
\begin{multline}\label{eq:reg1}
\frac{4\pi^2\tau(\chi)}{D}  \sum_{m=1} \frac{1}{Dm} \sum_{x(Dm)^{*}}  \overline{\chi(x)}e(\frac{x}{Dm})\sum_{k(Dm)}e(\frac{kDl\overline{x}}{Dm}) \sum_{c'\in \Z }  \int_{0}^{\infty}J_{k-1}(\frac{4\pi\sqrt{lD(Dmy-k)}}{Dm})\times  \\ \left[\int_0^\infty g(t) J_{k-1}(\frac{4\pi\sqrt{t(Dmy-k)}}{\sqrt{D}})dt\right] e(-c'y)dy.
\end{multline}

The $c'$-sum now clearly came from a Poisson summation, namely,\begin{multline}\label{eq:popo}
\sum_{c'\in \Z } \int_{0}^{\infty}J_{k-1}(\frac{4\pi\sqrt{lD(Dmy-k)}}{Dm})\times  \\ \left[\int_0^\infty g(t) J_{k-1}(\frac{4\pi\sqrt{t(Dmy-k)}}{\sqrt{D}})dt\right] e(-c'y)dy=\\ \sum_{c\in \Z} J_{k-1}(\frac{4\pi\sqrt{lD(Dmc'-k)}}{Dm})\int_0^\infty g(t) J_{k-1}(\frac{4\pi\sqrt{t(Dmc-k)}}{\sqrt{D}})dt.
\end{multline}

In order to check that  $$F(w)=J_{k-1}(\frac{4\pi\sqrt{lD(Dmw-k)}}{Dm})\int_0^\infty g(t) J_{k-1}(\frac{4\pi\sqrt{t(Dmw-k)}}{D})dt$$ satisfies the conditions for Poisson summation, we use the following lemma of \cite{KMV},
\begin{lemma}\label{mmm}
Let $h(x)$ be a smooth function supported in $[M,2M]$ which satisfies $|x^j h^{(j)}(x)| \ll (1+|\log x|)$
for all $i\geq 0, x >0.$ For $\nu$ complex and $j \geq 0$ we have $$\int_0^\infty J_{\nu}(x)h(x)dx \ll _{\nu,j} \frac{(1+|\log M|)}{M^{j-1}} \frac{M^{\Re \nu +j+1}}{(1+M)^{\Re \nu +j+1/2}}.$$

\end{lemma}

We apply this to the integral in $F(w)$ with $$h(t)=\frac{D^2}{16\pi^2(Dmw-k)^2}tg(\frac{D^2t^2}{16\pi^2(Dmw-k)^2}).$$ It is easy, but tedious, to check that the assumptions of the lemma are fulfilled by using the assumption on $g$ (in the hypothesis of Theorem \ref{mtheo}) that $|x^j g^{(j)}(x)| \ll (1+|\log x|).$ The lemma then gives $F(w) \ll \min(w^{k-1}, \frac{1}{w^j})$ for any $j >0$ for $w \in [0,\infty).$ So certainly Poisson summation holds in this case.

 Defining $Dmc-k=-j,$ \eqref{eq:reg1} again by regrouping equals \begin{multline}\label{eq:reg2}
\frac{4\pi^2\tau(\chi)}{D} \sum_{m=1} \frac{1}{Dm} \sum_{j \in \Z } \sum_{x(Dm)^{*}} \overline{\chi(x)}
e(\frac{x}{Dm})e(\frac{jDl\overline{x}}{Dm}) J_{k-1}(\frac{4\pi\sqrt{ljD}}{Dm})\int_0^\infty g(t) J_{k-1}(\frac{4\pi
\sqrt{tj}}{\sqrt{D}})dt=\\  \frac{4\pi^2\tau(\chi)}{D} \sum_{m=1} \frac{1}{Dm} \sum_{j \in \Z}\overline{\chi(j)} S_
{\chi}(Dl,j,Dm)J_{k-1}(\frac{4\pi\sqrt{lDj}}{Dm})\int_0^\infty g(t) J_{k-1}(\frac{4\pi\sqrt{tj}}{\sqrt{D}})dt=\\  \frac{4\pi^2\tau(\chi)}{D} \sum_{j \in \Z } \overline{\chi(j)} \left[ \sum_{m\equiv 0(D)} \frac{S_{\chi}(Dl,j,m)}{m}J_{k-1}(\frac{4\pi\sqrt{Dlj}}{m})\right] \int_0^\infty g(t) J_{k-1}(\frac{4\pi\sqrt{tj}}{\sqrt{D}})dt.\end{multline}

Now recall we ignored $g(l)$ from \eqref{eq:aft}, so \eqref{eq:mmm} equals  \begin{equation}\label{eq:nea}
g(l)+ \frac{2\pi i^{k}\tau(\chi)}{D} \sum_{j }\overline{\chi(j)} \left[ 2\pi i^{-k}\sum_{m\equiv 0(D)}\frac{ S_{\chi}(Dl,j,m)}{m}J_{k-1}(\frac{4\pi\sqrt{Dlj}}{m})\right] \int_0^\infty g(t) J_{k-1}(\frac{4\pi\sqrt{tj}}{\sqrt{D}})dt.
\end{equation}

The $g(l)$ term is again the diagonal term for the geometric side of the trace formula that comes from the term $$\sum_{f_D} a_{lD}(f_D)\overline{a_{lD}(f_D)}.$$  This is again using the fact that $|a_{D}(f_D)|=1.$

Now as $D$ is squarefree and $\chi$ is primitive, the space $B_k(N,\chi)$ is spanned by newforms which implies the Fourier coefficients are multiplicative in all the primes (including the bad primes) and  $|c_D(f)|=1.$
So using Petersson's formula again we get,  \begin{multline} \label{eq:al} \sum_{f \in B_k}\overline{a_{l}(f)}\left[ \frac{2\pi i^k\tau(\chi)}{D a_D(f)}\sum_{j} \overline{\chi(j)} a_j(f) \int_0^\infty g(x)J_{k-1}(\frac{4\pi\sqrt{jx}}{\sqrt{D}})dx\right]=\\
 \sum_{f \in B_k}\overline{a_{l}(f)}\left[ \frac{2\pi i^k\eta(f)}{\sqrt{D} }\sum_{j} \overline{\chi(j)} a_j(f) \int_0^\infty g(x)J_{k-1}(\frac{4\pi\sqrt{jx}}{\sqrt{D}})dx\right]=\\
 \sum_{f_D \in B_k}\overline{a_{l}(f_D)}\left[ \frac{2\pi i^k\eta(f)}{\sqrt{D} }\sum_{(j,D)=1}  a_j(f_D) \int_0^\infty g(x)J_{k-1}(\frac{4\pi\sqrt{jx}}{\sqrt{D}})dx\right].\end{multline}

Note however to show the connection to Voronoi summation from Theorem \ref{mic}, we need also the coefficients $a_j(f_D)$ with $(j,D)>1.$ 
 We prove \begin{lemma}
For $(l,D)=1,$
 \begin{equation}\label{eq:z0} \sum_{f_D \in B_k}\overline{a_{l}(f_D)}\left[ \frac{2\pi i^k\eta(f)}{\sqrt{D} }\sum_{(j,D)>1}  a_j(f_D) \int_0^\infty g(x)J_{k-1}(\frac{4\pi\sqrt{jx}}{\sqrt{D}})dx\right]=0.\end{equation}
\end{lemma}

\begin{proof}
First we write $j=D^kj', (j',D)=1.$ Using the definition of the coefficients $f_D(n)$ in Theorem \ref{mic},  the left hand side of \eqref{eq:z0} equals $$\frac{2\pi i^k}{\sqrt{D} }\sum_{k=1}^\infty \sum_{(j,D)=1} \overline{\chi(j)}  \sum_{f \in B_k}\overline{a_{lD^{k+1}}(f)}a_j(f) \int_0^\infty g(x)J_{k-1}(\frac{4\pi\sqrt{jD^kx}}{\sqrt{D}})dx.$$   Fix a $k,$ and following the same argument as we have previously, we apply Petersson's formula to get $$2\pi i^{-k}  \sum_{j} \overline{\chi(j)} \sum_{c=1}^\infty  \frac{S_{\chi}(j,lD^{k+1},Dc)}{Dc}J_{k-1}(\frac{4\pi\sqrt{nlD^{k+1}}}{Dc})\int_0^\infty g(x)J_{k-1}(\frac{4\pi\sqrt{jD^kx}}{\sqrt{D}})dx.$$ That we can apply Petersson's formula in this case follows from using the estimates of Lemma \ref{mmm}. With a change of variable in the Kloosterman sum this equals $$2\pi i^{-k}  \sum_{j} \sum_{c=1}^\infty  \frac{S_{\chi}(1,jlD^{k+1},Dc)}{Dc}J_{k-1}(\frac{4\pi\sqrt{nlD^{k+1}}}{Dc})\int_0^\infty g(x)J_{k-1}(\frac{4\pi\sqrt{jD^kx}}{\sqrt{D}})dx.$$ Interchanging the $j$- and $c$-sums justified by a similar Bessel function analysis as above, we apply Poisson summation to the $j$-sum modulo $Dc.$ The crucial arithmetic sums, analogous to obtaining \eqref{eq:forget}, are \begin{equation}\label{eq:lase}\sum_{x(Dc)^{*}} \overline{\chi(x)} e(\frac{x}{Dc}) \sum_{a(Dc)} e(\frac{\overline{x}aD^{k+1}l}{Dc})e(\frac{-am}{Dc}),\end{equation} where $m$ is the variable for  Poisson summation. The inner sum is non-zero only when $D^{k+1}l\equiv mx(Dc).$ If $(m,c)=1,$ then it is easy to check \eqref{eq:lase} is zero. As well if $D^h|m$ then $D^{h-1}|c$ as $(xl,D)=1$ for $h\leq k+1.$ So for a non-zero contribution we must have $D^{k+1}|m$ and $D^{k}|c.$ 
Writing $c=D^{k}c'$ and $m=D^{k+1}m',$ $x$ must satisfy  $l\equiv m'x(c').$ We can write these solutions as $x\equiv \overline{m'}l + c' b (Dc),$ where $b(D^{k+1}).$ So \eqref{eq:lase} equals $$Dc \sum_{b(D^{k+1})} \overline{\chi}( \overline{m'}l + c' b) e(\frac{\overline{m'}l + c' b}{Dc})=e(\frac{\overline{m'}l}{Dc})Dc  \sum_{b(D^{k+1})} \overline{\chi}( \overline{m'}l + c' b) e(\frac{ b}{D^{k+1}}) .$$ With a change of variables $b \to \overline{c}b,$ $b \to b-\overline{m'}l,$ the inner Gauss sum is $$\sum_{b(D^{k+1})} \overline{\chi}( b) e(\frac{ \overline{c'}b}{D^{k+1}})=\overline{\chi(c')}\sum_{b(D^{k+1})} \overline{\chi}( b) e(\frac{ b}{D^{k+1}}).$$ This last Gauss sum is zero as $\chi$ is a primitive character modulo $D$ and $k+1\geq 2.$
\end{proof}

\begin{remark}There is nothing special about the test function we used in the Lemma, and by a similar argument it easy to show for a ``nice" test function $V(x)$ and $k\geq 2,$ that $$\sum_{j=1}^\infty \overline{\chi(j)}  V(j) \sum_{f \in B_k}\overline{a_{D^{k}}(f)}a_j(f) =0.$$\end{remark}

 \end{proof}
 
 Applying the above lemma we get \eqref{eq:al} equaling $$ \sum_{f_D \in B_k}\overline{a_{l}(f_D)}\left[ \frac{2\pi i^k\eta(f)}{\sqrt{D} }\sum_{j=1}^\infty  a_j(f_D) \int_0^\infty g(x)J_{k-1}(\frac{4\pi\sqrt{jx}}{\sqrt{D}})dx\right],$$ which proves Theorem \ref{mtheo}.
 
 \section{Application of Hecke theory}\label{rogawski}
 Now to prove Corollary \ref{corr}. One can rewrite Theorem \ref{mtheo} as \begin{equation}\label{eq:comp} \frac{1}{2\pi i} \int_{(\sigma)} G(s) \left[\sum_{f \in B_k}\overline{a_l(f)}\bigg(L(f,s) - \frac{i^k\gamma(f,1-s)L(f,1-s)}{\gamma(f,s)}\bigg)\right] ds,
\end{equation} using that the Voronoi summation we take into consideration is equivalent to the functional equation.
Since \eqref{eq:comp} holds for any $g \in C_0^{\infty}(\R^{+})$ and in fact holds with slightly more care for all Schwarz functions, by completeness, it must hold $$\sum_{f \in B_k}\overline{a_l(f)}\bigg(L(f,s) - \frac{i^k\gamma(f,1-s)L(f,1-s)}{\gamma(f,s)}\bigg)=0.$$

Fix a form $f^{\circ} \in B_k.$ Now as $l$ was arbitrary and the space of forms $f \in B_k$ is finite dimensional, using the relation  $$a_{n}
(f)a_{l}(f)=\sum_{r|(n,l)}\chi(r)a_{\frac{nl}{r^2}}(f)$$ for $(nl,D)=1$ one can build a polynomial in the Hecke coefficients, call it $F(a_{q_1}(f),a_{q_2}(f),...,a_{q_N}(f)),$ such that $$\sum_{f \in B_k}F(a_{q_1}(f),a_{q_2}(f),...,a_{q_N}(f))\bigg(L(f,s) - \frac{i^k\gamma(f,1-s)L(f,1-s)}{\gamma(f,s)}\bigg)=0,$$  where $F$ equals $1$ for $f=f^{\circ},$ and equals $0$ for $f\neq f^{\circ}$ following \cite{H}. So we get a pointwise equality $$L(f,s) - \frac{i^k\gamma(f,1-s)L(f,1-s)}{\gamma(f,s)}=0,$$ which proves the corollary.

    \section{Appendix}
    We replicate Sarnak's argument from his letter to Langlands, \cite{S}. In order to do so, we use the Kuznetsov trace formula for the entire $GL_2$ spectrum. We refer to \cite{H} for the details. Let $H(D,\chi)$ denote the $GL_2$ spectrum with level $D$ and nebentypus $\chi.$
    \begin{theorem} Let $g,V \in C_0^\infty(\R^{+})$ with  $|x^j g^{(j)}(x)| \ll (1+|\log x|),$ $X$ a large fixed real number, $D$ and $\chi$ as above, then for any integer $A>0,$  
   \begin{equation} \label{eq:sarrk1} \sum_{n \leq X} \sum_{f \in H(D,\chi)}h(t_f,V) \overline{a_l(f)} a_n(f)g(n/X)=O(X^{-A})\end{equation}    
    \end{theorem}

    \begin{proof}
    We apply the Kuznetsov trace formula and Poisson summation again similar to getting \eqref{eq:forget} to get, 
     \begin{equation}\label{eq:forget1} \sum_{c=1}^\infty\frac{1}{Dc} \sum_{\substack{m\neq 0 \in \Z \\ (m,Dc)=1}} \overline{\chi(m)}e(\frac{-l\overline{m}}{Dc}) \int_{-\infty}^{\infty} g(\frac{t}{X}) V(\frac{4\pi \sqrt{t l}}{Dc}) e(\frac{-mt}{Dc})dt.\end{equation} Essentially, the argument only depends on showing the integral is bounded by $O(X^{-A}).$ Note as $V$ and $g$ are compactly supported, the $c$-sum is restricted to size $a\sqrt{X} \leq c \leq b\sqrt{X},$ for some absolute constants $a,b \in \R^{+},$ notated $c \sim \sqrt{X}.$ Note that $g^{(k)}(\frac{Dct}{X})\ll \frac{1}{X^{k/2}}$ and $V^{(h)}(\frac{4\pi \sqrt{t l}}{\sqrt{Dc}})\ll  \frac{1}{X^{h/2}}$ for $h,k \geq 0.$ Also the size of the integral is $\frac{X}{c} \sim \sqrt{X}.$ Using these estimates and integrating by parts $j$-times, after a change of variables $t \to Dct,$ it easy to check  \begin{equation}
Dc \int_{-\infty}^{\infty} g(\frac{Dct}{X}) V(\frac{4\pi \sqrt{t l}}{\sqrt{Dc}}) e(-mt)dt \ll \frac{Dc}{(\sqrt{X})^{j-1}m^j}\ll \frac{1}{(\sqrt{X})^{j-2}m^j}.
\end{equation}
So including the $c$- and $m$-sums we have $$\ll \frac{1}{(\sqrt{X})^{j-2}} \sum_{c \sim \sqrt{X}} \sum_{m} \frac{1}{m^j} \ll \frac{1}{(\sqrt{X})^{j-3}}.$$ Obviously, this implies the theorem by taking $\frac{j-3}{2} >A.$

    
    \end{proof}

\end{document}